\documentclass[12pt,reqno]{amsart}
\usepackage{amsmath}
\usepackage{amssymb}
\usepackage{amstext}
\usepackage{a4wide}
\usepackage{graphicx}
\allowdisplaybreaks \numberwithin{equation}{section}
\usepackage{color}
\usepackage{cases}
\usepackage{hyperref}

\numberwithin{equation}{section}

\newtheorem{theorem}{Theorem}[section]
\newtheorem{proposition}[theorem]{Proposition}

\newtheorem{lemma}[theorem]{Lemma}

\theoremstyle{definition}

\newtheorem{innercustomthm}{Lemma}
\newenvironment{customthm}[1]
{\renewcommand\theinnercustomthm{#1}\innercustomthm}
{\endinnercustomthm}

\newtheorem{definition}[theorem]{Definition}

\theoremstyle{remark}
\newtheorem{remark}[theorem]{Remark}

\begin{document}

\title[Expansion of Green's function for elliptic operators in divergence form]
{Expansion of Green's function and regularity of Robin's function for elliptic operators in divergence form}

\author{Daomin Cao and Jie Wan}
	
\address{Institute of Applied Mathematics, Chinese Academy of Sciences, Beijing 100190, and University of Chinese Academy of Sciences, Beijing 100049,  P.R. China}
\email{dmcao@amt.ac.cn}
\address{School of Mathematics and Statistics, Beijing Institute of Technology, Beijing 100081,  P.R. China}
\email{wanjie@bit.edu.cn}

\begin{abstract}
We consider Green's function $ G_K $ of the elliptic operator in divergence form $ \mathcal{L}_K=-\text{div}(K(x)\nabla  ) $ on a
bounded smooth domain $ \Omega\subseteq\mathbb{R}^n (n\geq 2) $ with  zero Dirichlet boundary condition, where $ K $ is a smooth positively definite matrix-valued function on $ \Omega $. We obtain a high-order asymptotic expansion of $ G_K(x, y) $, which defines uniquely a regular  part $ H_K(x, y) $. Moreover, we prove that the associated Robin's function $ R_K(x) = H_K(x, x) $ is smooth in $ \Omega $, despite the regular part $ H_K\notin C^1(\Omega\times\Omega)  $ in general.\\


\end{abstract}
	
\maketitle

\section{Introduction and main results}
Our purpose in this article is to study the expansion of Green's function and regularity of Robin's function for second-order elliptic operators in divergence form in all dimensions. Let $ \Omega\in\mathbb{R}^n $ be a bounded smooth domain with $ n\geq 2 $. Consider the following elliptic operator:
\begin{equation*}
\mathcal{L}_Ku=-\text{div}(K(x)\nabla u),
\end{equation*}
where $ K =(K_{i,j})_{2\times2} $ is a positively definite matrix satisfying
\begin{enumerate}
	\item[($\mathcal{K}$1).] $ K_{i,j} \in C^{\infty}(\overline{\Omega}) $ for $ 1\le i,j\le 2; $
	\item[($\mathcal{K}$2).] There exist constants $ \Lambda_1,\Lambda_2>0 $ such that $$ \Lambda_1|\zeta|^2\le (K(x)\zeta|\zeta) \le \Lambda_2|\zeta|^2,\ \ \ \ \forall\ x\in \Omega, \ \zeta\in \mathbb{R}^2.$$
\end{enumerate}
Let $ G_K(x,y) $ be the Green's function of $ \mathcal{L}_K $ associated to zero Dirichlet boundary condition, i.e., for any $ y\in\Omega$,
\begin{equation}\label{Def of Green1}
\begin{cases}
\mathcal{L}_KG_K(x,y)=\delta_y\ \ &\text{in}\ \Omega,\\
G_K(x,y)=0\ \ &\text{on}\ \partial\Omega.
\end{cases}
\end{equation}
Here $ \delta_y $ is the Dirac measure centered at $ y $. Multiplying test functions on both sides of the first equation of \eqref{Def of Green1} and integrating on $ \Omega $, we get the equivalent  characterization of Green's function in integral form, that is, for any $ u\in C^{2}(\Omega)\cap C^1_0(\overline{\Omega}) $
\begin{equation}\label{def of Green 2}
u(y)=\int_\Omega G_K(x,y)\mathcal{L}_Ku(x)dx,\ \ \forall y\in\Omega.
\end{equation}
Here $ C^{2}(\Omega)\cap C^1_0(\overline{\Omega}) $ is the space of all functions which are twice continuously differentiable and equal 0 on the boundary $ \partial\Omega $. The well-posedness of Green's function $ G_K $ defined by \eqref{Def of Green1} is considered in many articles. When the dimension $ n\geq 3 $, it is well known that Green's function of $ \mathcal{L}_K $ exists and is unique, see \cite{LSW63}.   Similar results is obtained in \cite{KN85} when the dimension  $ n=2 $. More results can be found in \cite{CL,DK1, GT,GW82,MM,TKB13} for instance.

The asymptotic expansion of Green's function of $ \mathcal{L}_K $, meanwhile, 
has also been widely concerned in recent decades. The understanding of the expansion of Green's function plays an essential role in many fields, especially in the study of concentration phenomena for some fluid mechanics models and semilinear elliptic equations, see, e.g., \cite{BF80,CF,CPY,De2,FB,R,WYZ,T}. For many concentration phenomena, the location of possible singularities of solutions is always determined by the corresponding Green's function, and the solutions are always perturbations of Green's function.  
A typical example is the  construction of concentrated vortex solutions to two dimensional, three dimensional axisymmetric and three dimensional helical symmetric incompressible Euler equations. For  planar Euler equations, one always gets concentrated vortex solutions by solving an elliptic problem
\begin{equation*}
-\Delta u=\lambda f(u)\ \ \text{in}\ \Omega;\ \ u=0\ \ \text{on}\ \partial\Omega,
\end{equation*}
and proving the asymptotic behavior of solutions $ u_\lambda $ as $ \lambda $ tends to infinity, see e.g., \cite{CF, CPY,DDMW,Nor,SV}. The limiting location is determined by the Kirchhoff-Routh function (see \cite{Lin}), which is a combination of Green's function of $ -\Delta $ and the  corresponding Robin's function. For the construction of concentrated vortex solutions to 3D axisymmetric Euler equations (called vortex rings) and 3D Euler equations with helical symmetry, corresponding elliptic operators appeared in elliptic problems are   $ -\frac{1}{a(x)}\text{div}(a(x)\nabla) $ and $ -\text{div}(K_H(x)\nabla) $ for some positive function $ a $ and positively definite matrix $ K_H $ respectively, see e.g., \cite{CW2,CW,De2,DV,FB}. When constructing a family of vortex solutions concentrating near several points, one always need to use the asymptotic expansion of the corresponding Green's function for these elliptic operators. Note that the   $ C^1 $ regularity of the corresponding Robin's function is necessary when constructing concentrated solutions by  perturbation arguments, see \cite{CPY,CW3,WYZ} for example.

The most common case is $ K\equiv Id_{n\times n} $, i.e., the  identity matrix of order $ n $, and then  the operator $\mathcal{L}_K $ is the standard Laplacian. In this case the associated Green's function $ G_\Omega $ of $ -\Delta $ in a domain $ \Omega $ with Dirichlet condition is of the decomposition (see \cite{GT})
\begin{equation}\label{Green of Laplacian}
G_\Omega(x,y)=\Phi_0(x-y)-H_{\Omega}(x,y),
\end{equation}
where $ \Phi_0 $ is the fundamental solution of $ -\Delta $ defined by
\begin{equation}\label{def of phi0}
\Phi_0(x)=\begin{cases}
-\frac{1}{2\pi}\log |x|,\ \ &\text{for}\ n=2;\\
\frac{1}{n(n-2)w_n}|x|^{2-n},\ \ &\text{for}\ n\geq 3;
\end{cases}
\end{equation}
with $ w_n $ the Lebesgue measure of the unit ball in $ \mathbb{R}^n $, and $ H_{\Omega} $ is the regular part of $ G_\Omega $ satisfying for any $ y\in\Omega $
\begin{equation}\label{def of H}
\begin{cases}
-\Delta H_{\Omega}(x,y)=0,\ \  &x\in\Omega,\\
H_{\Omega}(x,y)=\Phi_0(x-y),\ \ &x\in\partial\Omega.
\end{cases}
\end{equation}
Clearly $ H_{\Omega} $ is a function determined by $ \Omega $. Moreover, $ H_{\Omega} $ is smooth and symmetric in $ \Omega\times\Omega $, i.e.,  $ H_{\Omega}(x,y)=H_{\Omega}(y,x) $. Thus $ R_\Omega(x)=H_{\Omega}(x,x) $, called the Robin's function, is smooth in $ \Omega. $

When $ K$  is a positively definite constant coefficient matrix $ K_0 $,  we can also get the expansion of associated Green's function.
Let $ T_0 $ be the  positively definite matrix satisfying
$$ T_0^{-1}T_0^{-t}=K_0, $$
where $ T_0^{-t} $ is the transpose of $ T_0^{-1}. $
Denote $ T_0\Omega=\{T_0x\mid x\in\Omega\} $.
Using the coordinate transformation (see Lemma \ref{lemA-1} in Appendix), we find that Green's function $ G_{K_0} $ of $ \mathcal{L}_{K_0} $ has the decomposition
 \begin{equation*}\label{Green of K0}
 G_{K_0}(x,y)=\sqrt{\det K_0}^{-1}\Phi_0(T_0(x-y))-\sqrt{\det K_0}^{-1}H_{T_0\Omega}(T_0x,T_0y)\ \ \forall x,y\in\Omega, x\neq y,
 \end{equation*}
where $ H_{T_0\Omega}: T_0\Omega\times T_0\Omega\to\mathbb{R} $   satisfies for any $ y'\in T_0\Omega $
\begin{equation*}
\begin{cases}
-\Delta H_{T_0\Omega}(x',y')=0,\ \  &x'\in T_0\Omega,\\
H_{T_0\Omega}(x',y')=\Phi_0(x'-y'),\ \ &x'\in\partial T_0\Omega.
\end{cases}
\end{equation*}
Clearly the corresponding Robin's function
 $ R_{T_0\Omega}(x)=\sqrt{\det K_0}^{-1}H_{T_0\Omega}(T_0x,T_0x) $  is also smooth in $ \Omega. $

 When $ K $ is a matrix-valued function rather than a constant coefficient matrix, the situation turns out to be less clear.  If $ K $ is a diagonal matrix with the same diagonal elements, i.e., $ K(\cdot)=a(\cdot)Id_{n\times n} $ for some positive smooth function $ a $, \cite{Ye} constructed the expansion  of Green's function and proved the smoothness of the corresponding Robin's function. For any $ l\in \mathbb{N} $, it was proved in \cite{Ye} that Green's function $ G_a $ of $ -\Delta_a=-\frac{1}{a(x)}\text{div}(a(x)\nabla) $ has the high-order expansion
\begin{equation*}
G_a(x,y)=\Phi_0(x-y)+ \sum_{k=1}^{n+l-2}\Phi_k(x-y)+H^l_a(x,y)\ \ \text{in}\ \overline{\Omega}\times\Omega,
\end{equation*}
where $ H_a^l(x,y)\in C^{l,\gamma}(\overline{\Omega}\times\Omega) $, see also \cite{WYZ}. 
An interesting phenomena is,  for any $ n\geq2 $ if $ \nabla a(y)\neq 0, $ then $ x\mapsto H^0_a(x,y) $ does not belong to $ H^2(\Omega) $ since $   -\Delta_aH^0_a(x,y)\notin L^2(\Omega) $. However, $ H^0_a(x,x) $ is smooth in $ \Omega $, see Proposition 2.7 in \cite{Ye}.
For the case $ K $ being a general positively definite matrix-valued function, 
it follows from \cite{CW}   that Green's function  $ G_K $ has the following structure when $ n=2 $ (see \cite{CW}, theorem 1.2)
\begin{equation}\label{1001}
G_K(x, y)=\frac{\sqrt{\det K(x)}^{-1}+\sqrt{\det K(y)}^{-1}}{2}\Phi_0\left (\frac{T_x+T_y}{2}(x-y)\right )+S_K(x,y),
\end{equation}
where $ S_K\in C^{0,\gamma}(\Omega\times \Omega) $ for all $ \gamma\in(0,1) $ and $ S_K(x,y)=S_K(y,x) $ for  $ x,y \in \Omega. $ 
See also \cite{CW2}. This implies that the  regular part $ S_K $ of $ G_K $  is  just $ C^{0,\gamma} $, not smooth.
 Note that \eqref{1001} does not give us a high-order expansion of Green's function $ G_K(x, y) $, nor  
 proves whether the corresponding Robin's function $ S_K(x,x) $ is smooth. For the dimension $ n\geq 3 $, 
such results are also unknown for our knowledge. In this article, we intend to study this aspect. For the cases that dimension $ n\geq 2 $, we construct a high-order asymptotic expansion of Green's function $ G_K $ of \eqref{Def of Green1}, which permits us to define uniquely a regular part $ H_K $.
Then we prove the smoothness of the corresponding Robin's function $ R_K(x)=H_K(x,x) $. We find that $ R_K\in C^{\infty}(\Omega) $, although $ H_K $ does not belong  to $ C^{1}(\Omega\times \Omega) $ in general. Note that  the fundamental solution of $ -\Delta $ for $ n=2  $ has a logarithmic term,  which is different from that of  $ n\geq 3 $. Therefore, we get results by dealing with the odd and even cases of $ n $ separately.



Before stating our main results, let us first introduce some notations. Let $ n\geq 2 $ and
define $ T_x $  a $ C^\infty $ positively definite matrix-valued function determined by $ K $ satisfying
\begin{equation}\label{matrix T1}
T_x^{-1}(T_x^{-1})^{t}=K(x)\ \ \ \ \forall x\in \Omega.
\end{equation}
Clearly 
such $ T_x $ exists and is unique. For any multi-index $ \alpha=(\alpha_i)\in \mathbb{N}^n$ and $ x\in\mathbb{R}^n $, we denote $ |\alpha|=\sum_{i=1}^{n}\alpha_i $ and $ x^\alpha=\Pi_{i=1}^nx_i^{\alpha_i} $.  We define a linear space $ E^{n+2m}_{k+2m} $ as follows
\begin{definition}\label{def of E}
	Given $ k\in\mathbb{N}^*, m\in \mathbb{N}$, denote
	\begin{equation*}
	E^{n+2m}_{k+2m}=\text{span}\left \{\frac{x^\alpha}{r^{n+2m}}\mid |\alpha|=k+2m, \alpha\in \mathbb{N}^n\right \},
	\end{equation*}
	where $ r=||x|| $ is the classical Euclidean norm of $ x $.
	
\end{definition}

Our first result is on the expansion of Green's function of \eqref{Def of Green1} when $ n $ is odd.
\begin{theorem}\label{thm1}
Let $ n\geq 3 $ be odd and $ \gamma\in(0,1) $ be an arbitrary constant. Then for any $ l\in\mathbb{N} $, there exists a unique $ \Phi_i\in E^{n+2(2i-1)}_{i+2+2(2i-1)}$ for $  i=1,\cdots, n+l-2 $ depending on $ y\in\Omega $ and $ H^l(x,y)=H^l_y(x)\in C\left (\Omega, C^{l,\gamma}\left (\overline{\Omega}\right )\right )\cap C^{l}(\Omega\times\Omega)$ such that
	\begin{equation*}
	G_K(x,y)=\sqrt{\det K(y)}^{-1}\Phi_0\left (T_y(x-y)\right )+\sum_{i=1}^{n+l-2}\Phi_i\left (T_y(x-y)\right )+H^l(x,y)\ \ \text{in}\ \overline{\Omega}\times\Omega.
	\end{equation*}
\end{theorem}
\begin{remark}
Note that from Theorem \ref{thm1}, each term in the expansion of Green's function $ G_K $  is determined by $ \Phi_k\left (T_y(x-y)\right ) $, which is different from the term  $ \Phi_k(x-y) $ in the expansion of Green's function of $ -\Delta $ and $ -\Delta_a=-\frac{1}{a(x)}\text{div}\left (a(x)\nabla\right ) $, see \cite{GT,Ye}. However, if we choose $ K $ to be $Id_{n\times n}  $ and $ a Id_{n\times n} $, the expansion of $ G_K $ in Theorem \ref{thm1} coincides with the classical results in \cite{GT,Ye}.
\end{remark}
Concerning the cases with even dimension, since the fundamental solution of $ -\Delta $ in dimension two is $ -\frac{1}{2\pi}\log|x| $, we need to introduce some notations about singular and logarithmic terms.
Let $ n\in 2\mathbb{N}^*, k\in\mathbb{N}^*,m \in\mathbb{N}  $ and $ E^{n+2m}_{k+2m} $ be defined as in Definition \ref{def of E}. Let $ \mathbb{R}[x] $ be the set of real polynomials with variables $ x_i $. When $ k\geq n, $ we denote by $ E^{n+2m, s}_{k+2m} $ the singular set of $ E^{n+2m}_{k+2m} $, i.e., $ E^{n+2m,s}_{k+2m}= E^{n+2m}_{k+2m}\backslash \mathbb{R}[x] $.
 Denote
$$ L_m=\text{span}\{x^\alpha\log r\mid |\alpha|=m, \alpha\in\mathbb{N}^n\}. $$
We define $ F^{n+2m}_{k+2m} $ as
	\begin{equation*}
	F^{n+2m}_{k+2m}=
	\begin{cases}
	E^{n+2m}_{k+2m}\ &k < n,\\
	E^{n+2m,s}_{k+2m}\oplus L_{k-n}\ &k\geq  n.
	\end{cases}
	\end{equation*}
	
When the dimension $ n $ is even we have the following result concerning the expansion of Green's function.
\begin{theorem}\label{thm2}
	Let $ n\in 2\mathbb{N}^* $ and $ \gamma\in(0,1) $ be an arbitrary constant. Then for any $ l\in\mathbb{N} $, there exists a unique $ \Phi_i\in F^{n+2(2i-1)}_{i+2+2(2i-1)}$ for $  i=1,\cdots, n+l-2 $ depending on $ y\in\Omega $ and $ H^l(x,y)=H^l_y(x)\in C\left (\Omega, C^{l,\gamma}\left (\overline{\Omega}\right )\right )\cap C^{l}(\Omega\times\Omega)$ such that
	\begin{equation*}
	G_K(x,y)=\sqrt{\det K(y)}^{-1}\Phi_0\left (T_y(x-y)\right )+\sum_{i=1}^{n+l-2}\Phi_i\left (T_y(x-y)\right )+H^l(x,y)\ \ \text{in}\ \overline{\Omega}\times\Omega.
	\end{equation*}
\end{theorem}
\begin{remark}
 Compare with Theorem \ref{thm1},  the only difference is $ \Phi_i\in F^{n+2(2i-1)}_{i+2+2(2i-1)}$, rather than $ E^{n+2(2i-1)}_{i+2+2(2i-1)} $. Note also that when  taking $ n=2 $ and $ l=0 $, Theorem \ref{thm2} yields that
\begin{equation*}
G_K(x,y)=\sqrt{\det K(y)}^{-1}\Phi_0(T_y(x-y))+H^0(x,y),
\end{equation*}
where $ H^0(x,y)\in C^{0,\gamma}(\overline{\Omega}\times\Omega) $. This expansion coincides with   \eqref{1001} in \cite{CW}, since in this case $ S_K(x,y)=H^0(x,y)+\sqrt{\det K(y)}^{-1}\Phi_0(T_y(x-y))-\frac{\sqrt{\det K(x)}^{-1}+\sqrt{\det K(y)}^{-1}}{2}\Phi_0\left (\frac{T_x+T_y}{2}(x-y)\right )$ belongs to $ C^{0,\gamma} (\Omega\times\Omega) $.
\end{remark}

Our last result is on the smoothness of Robin's function, regardless of the parity of $ n $. Let us define the regular part of $ G_K $ and the associated Robin's function $ R_K $
\begin{equation*}
H_K(x,y)=H^0(x,y)\ \ \forall x ,y\in\Omega,
\end{equation*}
\begin{equation*}
R_K(x)=H_K(x,x)\ \ \forall x\in\Omega.
\end{equation*}
It follows from Theorem \ref{thm1} and Theorem \ref{thm2} that $ H_K\in C\left (\Omega, C^{0,\gamma}\left (\overline{\Omega}\right )\right )\cap C(\Omega\times\Omega) $. Note that this regularity is optimal since generally $ \mathcal{L}_KH_K(\cdot,y)\notin L^2(\Omega) $ and thus $ H_K(\cdot,y) $ does not belong to $ H^2(\Omega) $ for $ n\geq 2 $.
However, we have
\begin{theorem}\label{thm3}
Let $ n\geq 2 $ be an integer, $ \Omega $ be a bounded smooth domain and $ K $ be a positively definite smooth matrix-valued function in $ \Omega $. Then the Robin's function $ R_K(\cdot)\in C^\infty(\Omega) $.
\end{theorem}

\begin{remark}

Note that for $ n=2 $,  it is proved in \cite{CW} that Green's function $ G_K $ has a structure \eqref{1001}, where the regular part $ S_K $ belongs to $ C^{0,\gamma} $. Since $ S_K(x,x)=R_K(x) $ for any $ x\in\Omega $, using Theorem \ref{thm3}, we  actually show that $ S_K(x,x) $ is not just $ C^{0,\gamma} $, but smooth indeed, which improves the regularity results of the regular part $ S_K(x,x) $ in \cite{CW2,CW}.

\end{remark}

The paper is organized as follows. In section 2, we prove the expansion of $ G_K $ when $ n $ is odd. By similar strategy we prove the even dimensional case in section 3. Based on  Theorem \ref{thm1} and \ref{thm2}, we finish the proof of Theorem \ref{thm3} at the end of section 3.

\section{Odd dimensional case}

Now we begin to prove the expansion of Green's function $ G_K $ of \eqref{Def of Green1} when $ n\geq 3 $ is odd.  For any subset $ A_j $ in a vector space $ V $, we denote $ A_0+A_1+\cdots+A_p=\{\sum_{j=0}^pv_j\mid v_j\in A_j\}. $ We write $ A_0\oplus A_1\oplus\cdots\oplus A_p $ when the sum is direct.
For any $ l\in\mathbb{N}, 2l\leq k+2m $, we define
\begin{equation*}
E^{n+2m}_{k+2m,l}=\text{span}\left \{\frac{x^\alpha}{r^{n+2m-2l}}\mid |\alpha|=k+2m-2l, \alpha\in \mathbb{N}^n\right \}.
\end{equation*}
Note that from Definition \ref{def of E}, for any $ k\in\mathbb{N}^* $, $  E^{n+2m}_{K+2m}\in L^1_{loc}(\mathbb{R}^n) $  and for any $ m\in\mathbb{N} $, $ i,j\in\mathbb{N}^*, i\neq j $
\begin{equation*}
E^{n+2m}_{i+2m}\cap E^{n+2m}_{j+2m}=\{0\}.
\end{equation*}
Moreover, for any $ m,i\in\mathbb{N}^* $
\begin{equation*}
E^{n}_{i}\subsetneqq E^{n+2}_{i+2}\subsetneqq\cdots\cdots\subsetneqq E^{n+2m}_{i+2m}
\end{equation*}
and for any $ k\in\mathbb{N}^*, m,l\in \mathbb{N}, 2(l+1)\leq k+2m$
\begin{equation*}
E^{n+2m}_{k+2m,l+1}\subsetneqq E^{n+2m}_{k+2m,l}.
\end{equation*}

Using algebraic structure, we first show that $ E^{n+2m}_{k+2m} $ is contained in $ \Delta(E^{n+2m}_{k+2m+2}) $ for any $ k\in\mathbb{N}^*, m\in \mathbb{N}$.
\begin{lemma}\label{R of Laplacian }
Let $ n $ be odd. Then for any $ k\in\mathbb{N}^*, m\in \mathbb{N}$
\begin{equation*}
E^{n+2m}_{k+2m}\subseteq \Delta(E^{n+2m}_{k+2m+2}).
\end{equation*}
That is, for any $ f\in E^{n+2m}_{k+2m} $, there exists $ g\in E^{n+2m}_{k+2m+2} $ such that $ \Delta g=f. $
\end{lemma}

\begin{proof}
Note that $ E^{n+2m}_{k+2m}=\sum_{0\leq 2l\leq k+2m}E^{n+2m}_{k+2m,l} $. Given $ f(x)=\frac{x^\alpha}{r^{n+2m-2l}}\in E^{n+2m}_{k+2m,l} $ with $ |\alpha|=k+2m-2l\geq 0, $ let $ \bar{g}(x)=r^2f(x)=\frac{x^\alpha}{r^{n+2m-2(l+1)}}\in E^{n+2m}_{k+2m+2,l+1}. $ Then using the fact that $ x\cdot \nabla(x^\beta)=|\beta|x^\beta, \forall  \beta\in\mathbb{N}^n $, we obtain
\begin{equation*}
\Delta \bar{g}=\Delta(\frac{x^\alpha}{r^{n+2m-2(l+1)}})=\frac{\Delta x^\alpha}{r^{n+2m-2(l+1)}}+2(2l+2-n-2m)(|\alpha|+l-m)f.
\end{equation*}
Since $ n $ is odd, $ 2l+2-n-2m\neq 0. $ Since $ k\geq1, $ $ |\alpha|+l-m=\frac{k+|\alpha|}{2}>0. $ Thus  $ 2(2l+2-n-2m)(|\alpha|+l-m)\neq 0 $, which implies that
\begin{equation*}
\begin{split}
&\frac{1}{2(2l+2-n-2m)(|\alpha|+l-m)}\Delta \bar{g}-f\\
&\ \ =\frac{1}{2(2l+2-n-2m)(|\alpha|+l-m)}\frac{\Delta x^\alpha}{r^{n+2m-2(l+1)}}\in E^{n+2m}_{k+2m,l+1}.
\end{split}
\end{equation*}
In the case $ |\alpha|\leq 1, $  we already have $ \frac{1}{2(2l+2-n-2m)(|\alpha|+l-m)}\Delta \bar{g}-f=0 $, i.e., $ \frac{1}{2(2l+2-n-2m)(|\alpha|+l-m)} \bar{g} $ is a solution.
If $ |\alpha|> 1, $ then we conclude by decreasing induction on $ |\alpha| $ since $ \Delta x^\alpha\in \text{span}\{x^\beta\mid |\beta|=|\alpha|-2\}. $ Therefore, we prove the existence of $ g\in E^{n+2m}_{k+2m+2} $ such that $ \Delta g=f. $

\end{proof}

\begin{remark}
One computes directly that $ \Delta(E^{n+2m}_{k+2m+2})\subseteq E^{n+2m+2}_{k+2m+2}. $
Thus the above lemma shows that for any $ k\in\mathbb{N}^*, m\in \mathbb{N}$
\begin{equation*}
E^{n+2m}_{k+2m}\subseteq \Delta(E^{n+2m}_{k+2m+2})\subseteq E^{n+2m+2}_{k+2m+2}.
\end{equation*}
Moreover, the function $ g $ constructed in Lemma \ref{R of Laplacian } is indeed unique. Let $ g\in E^{n+2m}_{k+2m+2} $ such that $ \Delta g=0. $ Since $ n $ is odd, 
$ g $ cannot be in $ E^{n+2m}_{l}\backslash\{0\} $ for any $ l\in\mathbb{N}^* $. Thus the uniqueness of $ g $ holds.
\end{remark}

Let $ z\in\Omega $ be fixed and $ T_z $ be the  positively definite matrix defined by \eqref{matrix T1}. We define a linear space
\begin{equation*}
E^{n+2m}_{k+2m}(T_z)=\{f\circ T_z\mid f \in E^{n+2m}_{k+2m}\}.
\end{equation*}
Clearly, $ E^{n+2m}_{i+2m}(T_z)\cap E^{n+2m}_{j+2m}(T_z)=\{0\} $ for any $ m\in\mathbb{N} $, $ i,j\in\mathbb{N}^*, i\neq j $. A direct consequence of  Lemma \ref{R of Laplacian } is
\begin{lemma}\label{R of Laplacian2}
Let $ n $ be odd and $ z\in\Omega $ be fixed. Then for any $ k\in\mathbb{N}^*, m\in \mathbb{N}$
	\begin{equation*}
	E^{n+2m}_{k+2m}(T_z)\subseteq -\text{div}(K(z)\nabla)(E^{n+2m}_{k+2m+2}(T_z)).
	\end{equation*}
That is, for any $ f\in E^{n+2m}_{k+2m}(T_z) $, there exists $ g\in E^{n+2m}_{k+2m+2}(T_z) $ such that $ -\text{div}(K(z)\nabla) g=f. $
\end{lemma}

\begin{proof}
For any $ f\in E^{n+2m}_{k+2m}(T_z) $, we can find $ \hat{f}\in E^{n+2m}_{k+2m}  $ such that $ f=\hat{f}\circ T_z $. Hence by Lemma \ref{R of Laplacian }, there exists $ \hat{g}\in E^{n+2m}_{k+2m+2} $ such that $ -\Delta \hat{g}=\hat{f} $. Let $ g=\hat{g}\circ T_z $. One computes directly that $ g\in E^{n+2m}_{k+2m+2}(T_z) $ and $ -\text{div}(K(z)\nabla) g=f. $
\end{proof}

Using Lemma \ref{R of Laplacian2}, we can solve elliptic equations in divergence form  with constant coefficient matrix $ K(z) $ when $ f\in E^{n+2m}_{k+2m}(T_z) $.
\begin{lemma}\label{solu of Lap}
Let $ n $ be odd  and $ z\in\Omega $ be fixed. Let $ \Omega_0 $ be a bounded smooth domain in $ \mathbb{R}^n $ containing the origin $ 0 $ and $ c $ be a smooth function defined on $ \overline{\Omega}_0. $ Then for any $ k\in\mathbb{N}^*, m,l\in \mathbb{N}$, $ f\in E^{n+2m}_{k+2m}(T_z) $, there exists
\begin{equation}\label{101}
g\in E^{n+2m}_{k+2m+2}(T_z)\oplus E^{n+2m}_{k+2m+3}(T_z)\oplus\cdots\oplus E^{n+2m}_{n+2m+l}(T_z)\oplus  C^{l,\gamma}(\overline{\Omega}_0)
\end{equation}
such that
\begin{equation*}
-\text{div}(K(z)\nabla g)=c(x)f\ \ \text{in}\ \Omega_0;\ \ g=0\ \ \text{on}\ \partial\Omega_0.
\end{equation*}
Moreover, the expansion of $ g $ is unique.
\end{lemma}
\begin{remark}
If $ k\geq n+l-1 $, we just think of $ E^{n+2m}_{k+2m+2}(T_z)\oplus E^{n+2m}_{k+2m+3}(T_z)\oplus\cdots\oplus E^{n+2m}_{n+2m+l}(T_z) $ as $ \{0\}. $
\end{remark}

\begin{proof}
For $ x\in\mathbb{R}^n $, we denote the operator $ [x\cdot\nabla]=\sum_{i=1}^nx_i\partial_i $. It follows from the Taylor's formula that
\begin{equation*}
c(x)=\sum_{j=0}^{n+l-k}\frac{1}{j!}[x\cdot\nabla]^jc(0)+\frac{1}{(n+l-k)!}\int_{0}^{1}(1-t)^{n+l-k}[x\cdot\nabla]^{n+l-k+1}c(tx)dt.
\end{equation*}
Thus one has
\begin{equation*}
c(x)f\in E^{n+2m}_{k+2m}(T_z)\oplus E^{n+2m}_{k+2m+1}(T_z)\oplus\cdots\oplus E^{n+2m}_{n+2m+l}(T_z)\oplus  C^{l,\gamma}(\overline{\Omega}_0).
\end{equation*}
Using Lemma \ref{R of Laplacian2}, we get the existence of $ g_1\in E^{n+2m}_{k+2m+2}(T_z)\oplus E^{n+2m}_{k+2m+3}(T_z)\oplus\cdots\oplus E^{n+2m}_{n+2m+l+2}(T_z)\oplus  C^{l+2,\gamma}(\overline{\Omega}_0) $ such that $ -\text{div}(K(z)\nabla g_1)=f $.   Since the trace of any function in $ E^{n+2m}_{j}(T_z) $ on $ \partial\Omega_0 $ is smooth, there exists a  function $ g_2\in C^{\infty}(\overline{\Omega}_0) $ such that $ -\text{div}(K(z)\nabla g_2)=0 $ in $ \Omega_0 $ and $ g_2=g_1 $ on $ \partial \Omega_0 $. Note that $ E^{n+2m}_{n+2m+l+1}(T_z)\oplus E^{n+2m}_{n+2m+l+2}(T_z)\subseteq C^{l,\gamma}(\overline{\Omega}_0) $, we conclude that $ g=g_1-g_2 $ is the desired solution. The uniqueness of the decomposition of $ g $ follows directly from the maximum principle and the fact that the sum in \eqref{101} is direct.

\end{proof}

Based on Lemma \ref{solu of Lap}, we can
extend results of Lemma \ref{solu of Lap}  to the case of operators  $ -\text{div}(\tilde{K}(x)\nabla  ) $, when the coefficient matrix satisfies $ \tilde{K}(0)=K(z) $.

\begin{proposition}\label{solu of near Lap}
Let $ n $ be odd, $ z\in\Omega $, $ \Omega_0 $ and $ c  $ be as in Lemma \ref{solu of Lap}. Let $ \tilde{K}\in C^{\infty}(\overline{\Omega}_0, \mathbb{R}^{n\times n}) $ be a positively definite matrix-valued function with $ \tilde{K}(0)=K(z). $ Then for any $ k\in\mathbb{N}^*, m,l\in \mathbb{N}$, $ f\in E^{n+2m}_{k+2m}(T_z) $, there exists
\begin{equation*}
\zeta\in E^{n+2m}_{k+2m+2}(T_z)\oplus E^{n+2m+4}_{k+2m+7}(T_z)\oplus\cdots\oplus E^{n+2m+4(n+l-k-2)}_{n+2m+l+4(n+l-k-2)}(T_z)\oplus  C^{l,\gamma}(\overline{\Omega}_0)
\end{equation*}
such that
\begin{equation*}
-\text{div}(\tilde{K}(x)\nabla \zeta)=c(x)f\ \ \text{in}\ \Omega_0;\ \ \zeta=0\ \ \text{on}\ \partial\Omega_0.
\end{equation*}
Moreover, the expansion of $ \zeta $ is unique.
\end{proposition}

\begin{remark}
	If $ k\geq n+l-1 $, we just think of $ E^{n+2m}_{k+2m+2}(T_z)\oplus E^{n+2m+4}_{k+2m+7}(T_z)\oplus\cdots\oplus E^{n+2m+4(n+l-k-2)}_{n+2m+l+4(n+l-k-2)}(T_z) $ as $ \{0\}. $
\end{remark}

\begin{proof}
Consider the case $ c\equiv 1 $. To find a solution $ \zeta $, we construct a sequence of $ \{\zeta_j\}_{j=0}^{n+l-k+1}. $ First, let   $ \zeta_0 $ satisfy
\begin{equation}\label{2001}
-\text{div}(\tilde{K}(0)\nabla \zeta_0)=f\ \ \text{in}\ \Omega_0;\ \ \zeta_0=0\ \ \text{on}\ \partial\Omega_0.
\end{equation}
It follows from Lemma \ref{solu of Lap} that
\begin{equation*}
\zeta_0\in E^{n+2m}_{k+2m+2}(T_z)\oplus E^{n+2m}_{k+2m+3}(T_z)\oplus\cdots\oplus E^{n+2m}_{n+2m+l+2}(T_z)\oplus  C^{l+2,\gamma}(\overline{\Omega}_0).
\end{equation*}
Thus using Taylor's formula, we have
\begin{equation*}
\begin{split}
\text{div}((\tilde{K}(x)-\tilde{K}(0))\nabla \zeta_0)&=\sum_{i,j=1}^n(\tilde{K}_{ij}(x)-\tilde{K}_{ij}(0))\partial_{ij}\zeta_0+\sum_{i,j=1}^n\partial_i\tilde{K}_{ij}(x)\partial_{j}\zeta_0\\
&\in E^{n+2m+4}_{k+2m+5}(T_z)\oplus E^{n+2m+4}_{k+2m+6}(T_z)\oplus\cdots\oplus E^{n+2m+4}_{n+2m+l+4}(T_z)\oplus  C^{l,\gamma}(\overline{\Omega}_0).
\end{split}
\end{equation*}
Let $ \zeta_1 $ satisfy
\begin{equation*}
-\text{div}(\tilde{K}(0)\nabla \zeta_1)=\text{div}((\tilde{K}(x)-\tilde{K}(0))\nabla \zeta_0)\ \ \text{in}\ \Omega_0;\ \ \zeta_1=0\ \ \text{on}\ \partial\Omega_0.
\end{equation*}
Using  Lemma \ref{solu of Lap} again, we obtain
\begin{equation*}
\zeta_1\in E^{n+2m+4}_{k+2m+7}(T_z)\oplus E^{n+2m+4}_{k+2m+8}(T_z)\oplus\cdots\oplus E^{n+2m+4}_{n+2m+l+6}(T_z)\oplus  C^{l+2,\gamma}(\overline{\Omega}_0),
\end{equation*}
which implies that
\begin{equation*}
\begin{split}
\text{div}((\tilde{K}(x)-\tilde{K}(0))\nabla \zeta_1)\in E^{n+2m+8}_{k+2m+10}(T_z)\oplus E^{n+2m+8}_{k+2m+11}(T_z)\oplus\cdots\oplus E^{n+2m+8}_{n+2m+l+8}(T_z)\oplus  C^{l,\gamma}(\overline{\Omega}_0)
\end{split}
\end{equation*}
and so on. Thus by a recursive method we conclude that there exists $ \zeta_{j} $ for $ j=1,\cdots, n+l-k+1 $ with
\begin{equation}\label{2002}
-\text{div}(\tilde{K}(0)\nabla \zeta_{j})=\text{div}((\tilde{K}(x)-\tilde{K}(0))\nabla \zeta_{j-1})\ \ \text{in}\ \Omega_0;\ \ \zeta_j=0\ \ \text{on}\ \partial\Omega_0
\end{equation}
and
\begin{equation*}
\zeta_{j}\in E^{n+2m+4j}_{k+2m+5j+2}(T_z)\oplus E^{n+2m+4j}_{k+2m+5j+3}(T_z)\oplus\cdots\oplus E^{n+2m+4j}_{n+2m+4j+l+2}(T_z)\oplus  C^{l+2,\gamma}(\overline{\Omega}_0).
\end{equation*}
In particular, $ \zeta_{n+l-k+1}\in  C^{l+2,\gamma}(\overline{\Omega}_0) $. Finally, we solve
\begin{equation}\label{2003}
-\text{div}(\tilde{K}(x)\nabla h)=\text{div}((\tilde{K}(x)-\tilde{K}(0))\nabla \zeta_{n+l-k+1})\ \ \text{in}\ \Omega_0;\ \ h=0\ \ \text{on}\ \partial\Omega_0.
\end{equation}
By the classical theory for elliptic equations, $ h\in C^{l+2,\gamma}(\overline{\Omega}_0). $

Denote $ \zeta=\sum_{i=0}^{n+l-k+1}\zeta_i+h\in E^{n+2m}_{k+2m+2}(T_z)\oplus E^{n+2m+4}_{k+2m+7}(T_z)\oplus\cdots\oplus E^{n+2m+4(n+l-k)}_{n+2m+4(n+l-k)+l+2}(T_z)\oplus  C^{l+2,\gamma}(\overline{\Omega}_0). $ By \eqref{2001}, \eqref{2002} and \eqref{2003}, one computes directly that
\begin{equation*}
-\text{div}(\tilde{K}(x)\nabla \zeta)=f\ \ \text{in}\ \Omega_0;\ \ \zeta=0\ \ \text{on}\ \partial\Omega_0.
\end{equation*}
Note that $ E^{n+2m+4(n+l-k-1)}_{n+2m+4(n+l-k-1)+l+1}(T_z)\oplus E^{n+2m+4(n+l-k)}_{n+2m+4(n+l-k)+l+2}(T_z)\subseteq  C^{l,\gamma}(\overline{\Omega}_0) $, therefore $ \zeta $ is the desired solution. For the case $ c(x) $ being a smooth function, we can use the Taylor's expansion as that in Lemma \ref{solu of Lap} to get the result. The proof is thus finished.

\end{proof}

We are now able to prove the expansion of Green's function when $ n $ is odd.

\textbf{Proof of Theorem \ref{thm1}:}
We extend  $ K(x):\Omega\to\mathbb{R}^{n\times n} $ to a positively definite  smooth matrix-valued function in $ \mathbb{R}^n $. This can be done since $ \Omega $ is smooth.  For any $ w\in \Omega $, define $ \Omega_w=\{x-w\mid x\in\Omega\} $ and $ T_w(\Omega)=\{T_w(x)\mid x\in\Omega\}. $ By the positively definiteness and smoothness of $ K $, we  can choose $ M>0 $ sufficiently large such that $ T_w(\Omega_w)\subseteq B_M(0) $ for all $ w\in\Omega. $

Now we give the expansion of $ G_K(x,y) $.  Note that $ \Phi_0\in E^n_2. $ For fixed $ y\in\Omega, $ we define $ \hat{\Phi}_0=\sqrt{\det K(y)}^{-1}\Phi_0\circ T_y\in E^n_2(T_y) $. Then   direct computation yields that $ -\text{div}(K(y)\nabla \hat{\Phi}_0)=\delta_0 $. Consider solutions $ \zeta_{y} $ of the following equations
\begin{equation}\label{111}
\begin{cases}
-\text{div}(\tilde{K}_y(x)\nabla\zeta_{y})=\text{div} ((\tilde{K}_y(x)-K(y))\nabla\hat{\Phi}_0), \ & x \in B_M(0),\\
\zeta_{y}=0,\ &x \in \partial B_M(0),
\end{cases}
\end{equation}
where the coefficient matrix $ \tilde{K}_y(\cdot)= K(\cdot+y) $. By assumptions, we have $ \tilde{K}_y\in C^{\infty}(\overline{B_M(0)},\mathbb{R}^{n\times n}) $ and $ \tilde{K}_y(0)=K(y) $. Moreover, one computes directly that for any $ l\in\mathbb{N} $
\begin{equation*}
\text{div}((\tilde{K}_y(x)-K(y))\nabla\hat{\Phi}_0)\in E^{n+2}_3(T_y)\oplus E^{n+2}_4(T_y)\oplus\cdots\oplus E^{n+2}_{n+2+l}(T_y)\oplus C^{l,\gamma}(\overline{B_M(0)}).
\end{equation*}
Therefore using Proposition \ref{solu of near Lap} with  $ z=y $, $ \Omega_0=B_M(0) $ and $ \tilde{K}=\tilde{K}_y $, we get the  unique solution $ \zeta_{y} $ of \eqref{111} such that
\begin{equation*}
\begin{split}
\zeta_{y}& =\sum_{i=1}^{n+l-2}\hat{\Phi}_i+h\\
&\in E^{n+2}_5(T_y)\oplus E^{n+6}_{10}(T_y)\oplus\cdots\oplus E^{n+2+4(n+l-3)}_{n+2+4(n+l-3)+l}(T_y)\oplus C^{l,\gamma}(\overline{B_M(0)}),
\end{split}
\end{equation*}
where $ \hat{\Phi}_i\in E^{n+2+4(i-1)}_{5i}(T_y) $ for $ i=1,\cdots,n+l-2 $ and $ h\in C^{l,\gamma}(\overline{B_M(0)}) $.

Define $ \eta_y(x)=G_{K}(x,y)-\hat{\Phi}_0(x-y)-\zeta_{y}(x-y) $ for any $ x\in \Omega $. It follows from the definition of $ G_K $ that
\begin{equation*}
\mathcal{L}_KG_K(x,y)=\delta_{y}, \ x\in \Omega;\ \ \  G_{K}(x,y)=0, \ x\in\partial \Omega,
\end{equation*}
which implies that for any $ x\in \Omega $
\begin{equation}\label{2006}
\begin{split}
\mathcal{L}_K \eta_y(x)=&\text{div}((\tilde{K}_y(x)-K(y))\nabla\hat{\Phi}_0)(x-y)+\text{div}(\tilde{K}_y(x)\nabla \zeta_{y})(x-y)=0.
\end{split}
\end{equation}
Here we have used \eqref{111}. Note that \eqref{2006} also holds in sense of distribution.  Thus by Weyl's lemma, $ \eta_y \in C^\infty(\overline{\Omega}) $ since it is smooth on $ \partial\Omega $. To conclude,  we have
\begin{equation}\label{2007}
G_{K}(x,y)=\hat{\Phi}_0(x-y)+\sum_{i=1}^{n+l-2}\hat{\Phi}_i(x-y)+H^l(x,y),
\end{equation}
where $ \hat{\Phi}_i\in E^{n+2+4(i-1)}_{5i}(T_y)$ for $  i=1,\cdots, n+l-2 $ and $ x\mapsto H^l(x,y)\in C^{l,\gamma}(\overline{\Omega})$. Define $ \Phi_i=\hat{\Phi}_i\circ T_y^{-1}\in E^{n+2+4(i-1)}_{5i} $. This combined with \eqref{2007} yields the expansion for $ G_K(x,y) $
\begin{equation*}
G_{K}(x,y)=\sqrt{\det K(y)}^{-1}\Phi_0(T_y(x-y))+ \sum_{i=1}^{n+l-2}\Phi_i(T_y(x-y))+H^l(x,y),
\end{equation*}
where $ x\mapsto H^l(x,y) \in C^{l,\gamma}(\overline{\Omega}) $.

We now prove the regularity of $ H^l $, i.e., $ H^l(x,y)=H^l_y(x)\in C\left (\Omega, C^{l,\gamma}\left (\overline{\Omega}\right )\right )\cap C^{l}( \Omega\times \Omega) $. Indeed, from the above construction of $ G_K $, we know that $ H^l(x,y) $ satisfies for all $ y\in\Omega $
\begin{equation}\label{100}
\begin{split}
\mathcal{L}_KH^l(x,y)=&\mathcal{L}_KG_{K}(x,y)-\mathcal{L}_K\hat{\Phi}_0(x-y)-\mathcal{L}_K\left( \sum_{i=1}^{n+l-2}\hat{\Phi}_i(x-y)\right) \\
=&\text{div} \left (\left (K(x)-K(y)\right )\nabla\hat{\Phi}_0(x-y)\right )+\text{div}\left (K(x)\nabla \left( \sum_{i=1}^{n+l-2}\hat{\Phi}_i(x-y)\right)\right ) \\
\in&E^{n+2+4(n+l-3)}_{n+1+4(n+l-3)+l}(T_y)(x-y)\oplus E^{n+2+4(n+l-3)}_{n+2+4(n+l-3)+l}(T_y)(x-y)\oplus C^{l,\gamma}(\overline{B_M(0)})(x-y).
\end{split}
\end{equation}
Here we denote $ E^{n+2+4(n+l-3)}_{n+1+4(n+l-3)+l}(T_y)(x-y)=\{f(x-y)\mid f\in E^{n+2+4(n+l-3)}_{n+1+4(n+l-3)+l}(T_y)\} $ and $ C^{l,\gamma}(\overline{B_M(0)})(x-y)=\{f(x-y)\mid f\in C^{l,\gamma}(\overline{B_M(0)})\} $. Moreover, it follows from $ G_K(\cdot,y)=0 $ on $ \partial\Omega $ that
\begin{equation*}
H^l(x,y)=-\hat{\Phi}_0(x-y)-\sum_{i=1}^{n+l-2}\hat{\Phi}_i(x-y) \ \ \ \  x\in\partial \Omega.
\end{equation*}
By the continuity of the right-hand side of \eqref{100} and the boundary condition with respect to $ y $ in $ C^{l-2,\gamma}(\Omega) $   and $ C^{l,\gamma}(\partial\Omega) $, respectively, we can get $ H^l(x,y) \in C(\Omega, C^{l,\gamma}(\overline{\Omega})) $ and thus $ \frac{\partial H^l(x,y) }{\partial x}, \frac{\partial^2 H^l(x,y) }{\partial x^2},\cdots,\frac{\partial^l H^l(x,y) }{\partial x^l}\in C(\Omega\times \Omega) $.

Similarly,  taking $ \nabla_y $ to both sides of \eqref{100}, we can check that $ \nabla_y H^l(x,y)\in C(\Omega, C^{l-1,\gamma}(\overline{\Omega})) $, which implies that $ \frac{\partial H^l(x,y) }{\partial y},\frac{\partial^2 H^l(x,y) }{\partial y\partial x},\cdots,\frac{\partial^l H^l(x,y) }{\partial y\partial x^{l-1}}\in C(\Omega\times \Omega) $. Moreover, we can get that for any integer $ k,m $ with $ 0\leq k\leq m\leq l $,
\begin{equation*}
 \frac{\partial^m H^l(x,y) }{\partial y^k\partial x^{m-k}}\in C(\Omega\times \Omega).
\end{equation*}
Thus we conclude that $ H^l(x,y) $ is a $ C^l $ function over $ \Omega\times \Omega $.

\qed

Note that the regular part of $ G_K $ and the associated Robin's function $ R_K $ are defined by $ H_K(x,y)=H^0(x,y) $ and $ R_K(x)=H_K(x,x)$ for any $ x ,y\in\Omega $.
Therefore by Theorem \ref{thm1}, for any $ l\in\mathbb{N}^* $
\begin{equation}\label{112}
H_K(x,y)=H^0(x,y)=\sum_{i=n-1}^{n+l-2}\Phi_i(T_y(x-y))+H^l(x,y),\ \ \forall x ,y\in\Omega.
\end{equation}

Now we can prove Theorem \ref{thm3} when $ n $ is odd.
\begin{proposition}\label{thm3-1}
Let $ n $ be odd, $ \Omega $ and $ K $ be as in Theorem \ref{thm3}. Then the Robin's function $ R_K(\cdot)\in C^\infty(\Omega) $.
\end{proposition}

\begin{proof}
For any $ l\in\mathbb{N}^* $, it follows from \eqref{112} that $ R_K(x)=H^l(x,x) $ since $ \Phi_i(0)=0 $ for any $ i\geq n-1 $. So $ R_K\in C^l(\Omega) $. By the arbitrariness of $ l $, we get the result.
\end{proof}

\section{Even dimensional case: proof of Theorems \ref{thm2} and  \ref{thm3}}

Now we prove the even dimensional case. Let $ n\in 2\mathbb{N}^* $. Clearly from the definition of $ E^{n+2m,s}_{k+2m} $, we have  $ E^{n+2m}_{k+2m}= E^{n+2m,s}_{k+2m}\oplus \mathbb{R}[x] $ for any $ k\geq n, m\in\mathbb{N}. $
For $ L_m $, we first have
\begin{lemma}\label{le2-1}
For any $ m\in\mathbb{N}, $ $ L_m\subseteq \Delta(L_{m+2}\oplus \mathbb{R}[x]) $.
\end{lemma}
\begin{proof}
We  prove it by a direct algebraic construction. Let $ l\in\mathbb{N}, $ $ 2l\leq m, $ and $ L_{m,l}=\text{span}\{x^\alpha r^{2l}\log r\mid |\alpha|=m-2l\} $. Then $ L_m=\sum_{0\leq 2l\leq m}L_{m,l}. $ For $ f=x^\alpha r^{2l}\log r\in L_{m,l} $, let $ g_1=r^2f\in L_{m+2,l+1} $, then one computes directly that
\begin{equation*}
\Delta g_1=r^{2l+2}\log r \Delta(x^\alpha)+(2|\alpha|+n+4l+2)x^\alpha r^{2l}+(2l+2)(2l+|\alpha|+n) f.
\end{equation*}
Let $ g_2  $ be a polynomial with order $ m+2 $ satisfying $ \Delta g_2=(2|\alpha|+n+4l+2)x^\alpha r^{2l}. $ Since $ (2l+2)(2l+|\alpha|+n)>0 $, we find that there exists $ g_3=\frac{1}{(2l+2)(2l+|\alpha|+n)}(g_1-g_2)\in L_{m+2}\oplus\mathbb{R}[x] $ such that
\begin{equation*}
\Delta g_3-f=\frac{1}{(2l+2)(2l+|\alpha|+n)}r^{2l+2}\log r \Delta(x^\alpha)\in L_{m,l+1}.
\end{equation*}
Thus we can get the result by decreasing induction on $ |\alpha|. $
\end{proof}

Let us  define
\begin{equation*}
\tilde{E}^{n+2m}_{k+2m}=
\begin{cases}
F^{n+2m}_{k+2m}\ &k < n,\\
F^{n+2m}_{k+2m}\oplus \mathbb{R}[x] \ &k\geq  n.
\end{cases}
\end{equation*}
Similar to Lemma \ref{R of Laplacian }, using the definition of $ F^{n+2m}_{k+2m} $ and $ \tilde{E}^{n+2m}_{k+2m} $, we get
\begin{lemma}\label{le2-2}
There holds
\begin{equation*}
\tilde{E}^{n+2m}_{k+2m}\subseteq \Delta (\tilde{E}^{n+2m}_{k+2m+2})\ \ \forall k\in\mathbb{N}^*,m \in\mathbb{N}.
\end{equation*}
\end{lemma}

\begin{proof}
If $ k< n-2 $, then $ \tilde{E}^{n+2m}_{k+2m}=E^{n+2m}_{k+2m} $, $ \tilde{E}^{n+2m}_{k+2m+2}=E^{n+2m}_{k+2m+2} $. The results follows from Lemma \ref{R of Laplacian }.

When $ k=n-2 $ or $ n-1 $, let $ f\in E^{n+2m}_{k+2m,l} $ with $ 0\leq l\leq \frac{n-2+2m}{2}  $. If $ l< \frac{n-2+2m}{2} $, we can repeat the construction in Lemma \ref{R of Laplacian } with decreasing induction on $ |\alpha| $ to get the results. The only one case we need to pay attention to is $ 2l=n-2+2m $, i.e., $ f=\frac{x^\alpha}{r^2} $ with $ |\alpha|=k-n+2\leq 1. $ Let $ g=x^\alpha\log r \in L_{k-n+2}\subseteq \tilde{E}^{n+2m}_{k+2m+2} $. One computes directly that $ \Delta g=(2|\alpha|+n-2)f. $ Note that $ 2|\alpha|+n-2> 0 $ for $ n\geq 4. $ For $ n=2, $ only the case $ |\alpha|=1 $  occurs since $ k\geq 1 $. Thus $ 2|\alpha|+n-2\ne 0 $ and we get the results.

For the case $ k\geq n, $ if $ f\in L_{k-n}\oplus \mathbb{R}[x] $, the result is shown by Lemma \ref{le2-1}. Let $ f\in E^{n+2m}_{k+2m, l} $ for some $ 0\leq 2l\leq k+2m $.  If   $ 2l< n-2+2m $, we can get the result by decreasing induction on $ |\alpha| $. If $ 2l\geq n+2m $, then $ f $ is a polynomial with order $ k-n $. So there exists a polynomial $ g\in\mathbb{R}[x]\subseteq \tilde{E}^{n+2m}_{k+2m+2} $ such that $ \Delta g=f. $ It remains to solve the case $ 2l=n-2+2m $, i.e., $ f=\frac{x^\alpha}{r^2} $ with $ |\alpha|=k-n+2. $ Let $ g_0=x^\alpha\log r\in L_{k+2-n}. $ Then
\begin{equation*}
\Delta g_0=(2|\alpha|+n-2)\frac{x^\alpha}{r^2}+\log r\times \Delta(x^\alpha),
\end{equation*}
which implies that $ \Delta(\frac{1}{2|\alpha|+n-2}g_0)-f\in L_{k-n} $. Since $ L_{k-n}\subseteq \Delta(L_{k-n+2}\oplus \mathbb{R}[x]) $, we get the existence of $ g\in\mathbb{R}[x]\subseteq \tilde{E}^{n+2m}_{k+2m+2} $ such that $ \Delta g=f. $

\end{proof}

Let $ z\in\Omega $ be fixed and $ T_z $ be the positively definite matrix defined by \eqref{matrix T1}. We define
\begin{equation*}
F^{n+2m}_{k+2m}(T_z)=\{f\circ T_z\mid f \in F^{n+2m}_{k+2m}\};\ \  \tilde{E}^{n+2m}_{k+2m}(T_z)=\{f\circ T_z\mid f \in \tilde{E}^{n+2m}_{k+2m}\}.
\end{equation*}
A direct consequence of Lemma \ref{le2-2} is
\begin{lemma}\label{le2-2-2}
	Let $ n $ be even and $ z\in\Omega $ be fixed. Then for any $ k\in\mathbb{N}^*, m\in \mathbb{N}$
	\begin{equation*}
	\tilde{E}^{n+2m}_{k+2m}(T_z)\subseteq -\text{div}(K(z)\nabla)(\tilde{E}^{n+2m}_{k+2m+2}(T_z)).
	\end{equation*}
\end{lemma}

Based on Lemma \ref{le2-2-2}, we get the existence of solutions to elliptic equations in divergence form  with constant coefficient matrix $ K(z) $ when $ f\in F^{n+2m}_{k+2m}(T_z) $.

\begin{lemma}\label{le2-3}
Let $ n $ be even  and $ z\in\Omega $ be fixed. Let $ \Omega_0 $ be a bounded smooth domain in $ \mathbb{R}^n $ containing the origin $ 0 $ and $ c $ be a smooth function defined on $ \overline{\Omega}_0. $ Then for any $ k\in\mathbb{N}^*, m,l\in \mathbb{N}$, $ f\in F^{n+2m}_{k+2m}(T_z) $, there exists
	\begin{equation*}
	g\in F^{n+2m}_{k+2m+2}(T_z)\oplus F^{n+2m}_{k+2m+3}(T_z)\oplus\cdots\oplus F^{n+2m}_{n+2m+l}(T_z)\oplus  C^{l,\gamma}(\overline{\Omega}_0)
	\end{equation*}
	such that
	\begin{equation*}
-\text{div}(K(z)\nabla g)=c(x)f\ \ \text{in}\ \Omega_0;\ \ g=0\ \ \text{on}\ \partial\Omega_0.
	\end{equation*}
Moreover, the expansion of $ g $ is unique.
\end{lemma}
\begin{proof}
Using Taylor's formula, we have
	\begin{equation*}
	c(x)f\in F^{n+2m}_{k+2m}(T_z)\oplus F^{n+2m}_{k+2m+1}(T_z)\oplus\cdots\oplus F^{n+2m}_{n+2m+l}(T_z)\oplus  C^{l,\gamma}(\overline{\Omega}_0).
	\end{equation*}
By Lemma \ref{le2-2-2}, we get the existence of $ \bar{g}_1\in \tilde{E}^{n+2m}_{k+2m+2}(T_z)\oplus \tilde{E}^{n+2m}_{k+2m+3}(T_z)\oplus\cdots\oplus \tilde{E}^{n+2m}_{n+2m+l+2}(T_z)\oplus  C^{l+2,\gamma}(\overline{\Omega}_0) $ such that $ -\text{div}(K(z)\nabla \bar{g}_1)=f $. So $ \bar{g}_1\in F^{n+2m}_{k+2m+2}(T_z)\oplus F^{n+2m}_{k+2m+3}(T_z)\oplus\cdots\oplus F^{n+2m}_{n+2m+l+2}(T_z)\oplus  C^{l+2,\gamma}(\overline{\Omega}_0) $. Let $ \bar{g}_2\in C^{\infty}(\overline{\Omega}_0) $ be a function  satisfying $ -\text{div}(K(z)\nabla \bar{g}_2)=0  $ in $ \Omega_0 $ and  $ \bar{g}_2=\bar{g}_1 $ on $ \partial \Omega_0 $. Since $ F^{n+2m}_{n+2m+l+1}(T_z)\oplus F^{n+2m}_{n+2m+l+2}(T_z)\subseteq C^{l,\gamma}(\overline{\Omega}_0) $, we conclude that $ g=\bar{g}_1-\bar{g}_2 $ is the desired solution.

\end{proof}

Then, we further extend  results of Lemma \ref{le2-3} to  elliptic operators in divergence form, which is close to $ -\text{div}(K(z)\nabla ) $ near the origin $ 0 $ when $ n $ is even.

\begin{proposition}\label{le2-4}
Let $ n $ be even, $ z\in\Omega $, $ \Omega_0 $    and $ c  $ be as in Lemma \ref{le2-3}. Let $ \tilde{K}\in C^{\infty}(\overline{\Omega}_0, \mathbb{R}^{n\times n}) $ be a positively definite matrix-valued function with $ \tilde{K}(0)=K(z). $ Then for any $ k\in\mathbb{N}^*, m,l\in \mathbb{N}$, $ f\in F^{n+2m}_{k+2m}(T_z) $, there exists
	\begin{equation*}
	\xi\in F^{n+2m}_{k+2m+2}(T_z)\oplus F^{n+2m+4}_{k+2m+7}(T_z)\oplus\cdots\oplus F^{n+2m+4(n+l-k-2)}_{n+2m+l+4(n+l-k-2)}(T_z)\oplus  C^{l,\gamma}(\overline{\Omega}_0)
	\end{equation*}
	such that
	\begin{equation*}
	-\text{div}(\tilde{K}(x)\nabla \xi)=c(x)f\ \ \text{in}\ \Omega_0;\ \ \xi=0\ \ \text{on}\ \partial\Omega_0.
	\end{equation*}
Moreover, the expansion of $ \zeta $ is unique.
\end{proposition}

\begin{proof}
The idea is similar to the proof of Proposition \ref{solu of near Lap} and we prove it here just for completeness. Consider the case $ c\equiv 1 $. To find a solution $ \xi $, we construct a sequence of $ \{\xi_j\}_{j=0}^{n+l-k+1}. $ First, let   $ \xi_0 $ satisfy
	\begin{equation*}
	-\text{div}(\tilde{K}(0)\nabla \xi_0)=f\ \ \text{in}\ \Omega_0;\ \ \xi_0=0\ \ \text{on}\ \partial\Omega_0.
	\end{equation*}
It follows from Lemma \ref{le2-3} that
	\begin{equation*}
\xi_0\in F^{n+2m}_{k+2m+2}(T_z)\oplus F^{n+2m}_{k+2m+3}(T_z)\oplus\cdots\oplus F^{n+2m}_{n+2m+l+2}\oplus  C^{l+2,\gamma}(\overline{\Omega}_0),
	\end{equation*}
from which we deduce,
	\begin{equation*}
	\begin{split}
	\text{div}((\tilde{K}(x)-\tilde{K}(0))\nabla \xi_0)&=\sum_{i,j=1}^n(\tilde{K}_{ij}(x)-\tilde{K}_{ij}(0))\partial_{ij}\xi_0+\sum_{i,j=1}^n\partial_i\tilde{K}_{ij}(x)\partial_{j}\xi_0\\
	&\in F^{n+2m+4}_{k+2m+5}(T_z)\oplus F^{n+2m+4}_{k+2m+6}(T_z)\oplus\cdots\oplus F^{n+2m+4}_{n+2m+l+4}(T_z)\oplus  C^{l,\gamma}(\overline{\Omega}_0).
	\end{split}
	\end{equation*}
	Let $ \xi_1 $ satisfy
	\begin{equation*}
	-\text{div}(\tilde{K}(0)\nabla \xi_1)=\text{div}((\tilde{K}(x)-\tilde{K}(0))\nabla \xi_0)\ \ \text{in}\ \Omega_0;\ \ \xi_1=0\ \ \text{on}\ \partial\Omega_0.
	\end{equation*}
	By Lemma \ref{le2-3}, $ \xi_1\in F^{n+2m+4}_{k+2m+7}(T_z)\oplus F^{n+2m+4}_{k+2m+8}(T_z)\oplus\cdots\oplus F^{n+2m+4}_{n+2m+l+6}(T_z)\oplus  C^{l+2,\gamma}(\overline{\Omega}_0). $
This yields that $ 	\text{div}((\tilde{K}(x)-\tilde{K}(0))\nabla \xi_1)\in F^{n+2m+8}_{k+2m+10}(T_z)\oplus F^{n+2m+8}_{k+2m+11}(T_z)\oplus\cdots\oplus F^{n+2m+8}_{n+2m+l+8}(T_z)\oplus  C^{l,\gamma}(\overline{\Omega}_0)  $ and so on.
Hence we can find $ \xi_{j} $ for $ j=1,\cdots, n+l-k+1 $ satisfying
	\begin{equation*}
	-\text{div}(\tilde{K}(0)\nabla \xi_{j})=\text{div}((\tilde{K}(x)-\tilde{K}(0))\nabla \xi_{j-1})\ \ \text{in}\ \Omega_0;\ \ \xi_j=0\ \ \text{on}\ \partial\Omega_0
	\end{equation*}
	and
	\begin{equation*}
	\xi_{j}\in F^{n+2m+4j}_{k+2m+5j+2}(T_z)\oplus F^{n+2m+4j}_{k+2m+5j+3}(T_z)\oplus\cdots\oplus F^{n+2m+4j}_{n+2m+4j+l+2}(T_z)\oplus  C^{l+2,\gamma}(\overline{\Omega}_0).
	\end{equation*}
	In particular, $ \xi_{n+l-k+1}\in  C^{l+2,\gamma}(\overline{\Omega}_0) $. Finally, we solve $ h\in C^{l+2,\gamma}(\overline{\Omega}_0)  $ satisfying
	\begin{equation*}
	-\text{div}(\tilde{K}(x)\nabla h)=\text{div}((\tilde{K}(x)-\tilde{K}(0))\nabla \xi_{n+l-k+1})\ \ \text{in}\ \Omega_0;\ \ h=0\ \ \text{on}\ \partial\Omega_0.
	\end{equation*}
Denote $ \xi=\sum_{i=0}^{n+l-k+1}\xi_i+h\in F^{n+2m}_{k+2m+2}(T_z)\oplus F^{n+2m+4}_{k+2m+7}(T_z)\oplus\cdots\oplus F^{n+2m+4(n+l-k)}_{n+2m+4(n+l-k)+l+2}(T_z)\oplus  C^{l+2,\gamma}(\overline{\Omega}_0). $ One computes directly that
	\begin{equation*}
	-\text{div}(\tilde{K}(x)\nabla \xi)=f\ \ \text{in}\ \Omega_0;\ \ \zeta=0\ \ \text{on}\ \partial\Omega_0.
	\end{equation*}
Thus $ \xi $ is the desired solution. For the case $ c $ being a smooth function, we can use the same Taylor's expansion as in Lemma \ref{le2-3} to get the result.

\end{proof}

\textbf{Proof of Theorem \ref{thm2}:}
Based on Proposition \ref{le2-4},  we can use similar method as that in Theorem \ref{thm1} to get the expansion of  Green's function $ G_K $ of the operator $ \mathcal{L}_K $ in case that $ n $ is even, just by replacing $ E^{n+2m}_{k+2m} $  with $ F^{n+2m}_{k+2m} $. We omit the proof here.
\qed

Similar to the odd dimensional case, we define the regular part of $ G_K $ by $ H_K(x,y)=H^0(x,y) $ for any $ x ,y\in\Omega $
and  Robin's function $ R_K(x)=H_K(x,x)$ for any $ x\in\Omega. $ By Theorem \ref{thm2}, for any $ l\in\mathbb{N}^* $
\begin{equation}\label{200}
H_K(x,y)= \sum_{i=n-1}^{n+l-2}\Phi_i(T_y(x-y))+H^l(x,y),\ \ \forall x ,y\in\Omega.
\end{equation}

Since $ \Phi_i(0)=0 $ for any $ i\geq n-1 $, we can get that $ R_K $ is smooth when $ n $ is even, which is summarized as follows.
\begin{proposition}\label{thm3-2}
Let $ n $ be even, $ \Omega $ and $ K $ be as in Theorem \ref{thm3}. Then  Robin's function $ R_K(\cdot)\in C^\infty(\Omega) $.
\end{proposition}

\textbf{Proof of Theorem \ref{thm3}:} Using Propositions \ref{thm3-1} and \ref{thm3-2}, we finish the proof of Theorem \ref{thm3}.
\qed

\begin{remark}
We give a final remark that the smoothness of $ K $ and $ \Omega $ is indeed used to obtain the smoothness of Robin's function $ R_K $. If we just need to get
	$ R_K\in C^l $ for some $ l $, the assumptions on $ K $ and $ \Omega $ can be  weakened.
	
\end{remark}

\section*{Appendix}
In this Appendix, we give proof of  the expansion of Green's function of $ \mathcal{L}_K $ when the coefficient matrix $ K $ is a constant coefficient matrix $ K_0 $.

\begin{customthm}{A.1}\label{lemA-1}
Let $ T_0 $ be the  positively definite matrix satisfying
 $$ T_0^{-1}T_0^{-t}=K_0.$$
 Denote $ T_0\Omega=\{T_0x\mid x\in\Omega\} $.
 Let $ G_{K_0} $ be the Green's function of $ \mathcal{L}_{K_0} $ in $ \Omega $ with zero Dirichlet boundary condition. Then
 \begin{equation}\label{2004}
 G_{K_0}(x,y)=\sqrt{\det K_0}^{-1}\Phi_0(T_0(x-y))-\sqrt{\det K_0}^{-1}H_{T_0\Omega}(T_0x,T_0y)\ \ \forall x,y\in\Omega, x\neq y,
 \end{equation}
 where $ H_{T_0\Omega}: T_0\Omega\times T_0\Omega\to\mathbb{R} $ is the regular part   satisfying for any $ y'\in T_0\Omega $
 \begin{equation}\label{2005}
 \begin{cases}
 -\Delta H_{T_0\Omega}(x',y')=0,\ \  &x'\in T_0\Omega,\\
 H_{T_0\Omega}(x',y')=\Phi_0(x'-y'),\ \ &x'\in\partial T_0\Omega.
 \end{cases}
 \end{equation}
	
\end{customthm}

\begin{proof}
For any $ u\in C^{2}(\Omega)\cap C^1_0(\overline{\Omega})  $ with $ f=\mathcal{L}_{K_0}u $,
let $ \bar{u}(x')=u(T_0^{-1}x')=u(x) $ and $ \bar{f}(x')=f(T_0^{-1}x')=f(x) $ for any $ x'\in T_0\Omega $. Here we set $ x'=T_0x $. Then one computes directly that $ -\Delta \bar{u} =\bar{f}  $ in $ T_0\Omega $ and $ \bar{u}\in C^{2}(T_0\Omega)\cap C^1_0(\overline{T_0\Omega}) $. From \eqref{Green of Laplacian}, we have
	\begin{equation}\label{1011}
	\bar{u}(y')=\int_{T_0\Omega}\left (\Phi_0(x'-y')-H_{T_0\Omega}(x',y')\right )\bar{f}(x')dx',\ \ \forall y'\in T_0\Omega,
	\end{equation}
where $ H_{T_0\Omega}: T_0\Omega\times T_0\Omega\to\mathbb{R} $ satisfies \eqref{2005}.
Using coordinate transformation, \eqref{1011} yields
	\begin{equation*}
	\begin{split}
	u(y)=&\int_{\Omega}\left (\Phi_0(T_0x-T_0y)-H_{T_0\Omega}(T_0x,T_0y)\right )f(x)\cdot\det T_0dx\\
	=&\sqrt{\det K_0}^{-1}\int_{\Omega}\left (\Phi_0(T_0x-T_0y)-H_{T_0\Omega}(T_0x,T_0y)\right )f(x)dx,\ \ \forall y\in \Omega.
	\end{split}
	\end{equation*}
This implies that $ G_{K_0} $ has decomposition \eqref{2004}. The proof is thus complete.
\end{proof}

\subsection*{Acknowledgments:}

\par
D. Cao was supported by  National Key R\&D
Program (2022YFA1005602) and NNSF of China(grant No.12371212).  J. Wan was supported by NNSF of China (grant No. 12101045).

\subsection*{Conflict of interest statement} On behalf of all authors, the corresponding author states that there is no conflict of interest.

\subsection*{Data availability statement} All data generated or analysed during this study are included in this published article  and its supplementary information files.

\end{document}